\documentclass[reqno,10pt]{article} 
\usepackage[dvips]{graphics} 
\usepackage{amsmath, fancybox, fancyhdr} 
\usepackage{fig4tex} 
\usepackage{theorem}
\usepackage{amssymb, amsmath, amsbsy, amsfonts}
\usepackage{graphics,graphicx}

\newtheorem{theorem}{Theorem}[section]
\newtheorem{lemma}{Lemma}[section]
\newtheorem{prop}{Proposition}[section]

\newtheorem{cor}{Corollary}[section]
{\theorembodyfont{\rmfamily}
\newtheorem{example}{Example}[section]
\newtheorem{remark}{Remark}[section]
}

\newenvironment{proof}{\noindent\text{\textbf{Proof.\:}}}{}

\newcommand{\N}{\mathbb{N}}
\newcommand{\R}{\mathbb{R}}
\newcommand{\T}{\mathbb{T}}
\newcommand{\C}{\mathcal{C}}
\newcommand{\Om}{\Omega}

\newcommand{\footnoteremember}[2]{\footnote{#2}\newcounter{#1}\setcounter{#1}{\value{footnote}}}

\def\qed{\hfill $\square$ \goodbreak \bigskip}
\def\eps{\varepsilon}

\title{Polygons as optimal shapes with convexity constraint}

\author{Jimmy Lamboley\footnoteremember{1}{ENS Cachan Bretagne, IRMAR, UEB, av Robert Schuman, 35170 Bruz, France}, Arian Novruzi\footnote{University of Ottawa, Department of Mathematics and Statistics, 585 King Edward, Ottawa, ON, K1N 6N5, Canada}}
\date{October 2008}

\begin{document}
\maketitle
\begin{abstract}
In this paper, we focus on the following general shape optimization problem:
$$
\min\{J(\Om),\ \Om\ convex,\ \Omega\in\mathcal S_{ad}\},
$$
where $\mathcal S_{ad}$ is a set of 2-dimensional admissible shapes and $J:\mathcal{S}_{ad}\rightarrow\R$ is a shape functional.
Using a specific parameterization of the set of convex domains, we derive some extremality conditions (first and second order) for this kind of problem.
Moreover, we use these optimality conditions to prove that, for a large class of functionals (satisfying a concavity like property), any solution to this shape optimization problem is a polygon.\\

\smallskip

\end{abstract}

\section{Introduction}
In this paper, we are mainly interested in questions related to a convexity constraint in shape optimization. We deal with the following general shape optimization problem: 
\begin{equation}\label{eq:Pb1}
\min\{J(\Om),\ \Om\ convex,\ \Omega\in\mathcal S_{ad}\},
\end{equation}
where $J$ is a shape functional defined on a class $\mathcal{S}_{ad}$ of subsets of $\R^2$.\\

Our goal is, on one hand, to write down explicit first and second order optimality conditions for general 2-dimensional shape optimization problems with convexity constraint and, on the other hand, to use them to exhibit a family of shape functionals for which optimal shapes are polygons.\\

As it is well-known, dimension 2 allows to write the convexity constraint through the positivity of a {\bf linear} operator with respect to the shape. More precisely, if one uses polar coordinates representation $(r,\theta)$ for the domains, namely
\begin{equation}\label{eq:Omega_u}
\Om_u:=\left\{(r,\theta)\in [0,\infty)\times\R\;;\;r<\frac{1}{u(\theta)}\right\},
\end{equation}
where $u$ is a positive  and $2\pi$-periodic function, then
$$\Om_u\textrm{ is convex }\Longleftrightarrow u''+u\geq 0.$$

As a consequence, we look at shape optimization problems of the form
\begin{equation}\label{eq:Pb2}
u_0\in \mathcal F_{ad}\;,\;j(u_0)=\min \{j(u):=J(\Om_u),\;u>0,\;u''+u\geq0,\;u\in \mathcal F_{ad}\}
\end{equation}
where $\mathcal F_{ad}$ is a set of convenient $2\pi$-periodic admissible functions.\\

A first contribution is to write down explicitly the first and second order optimality conditions on $u_0$ for some specific choices of $\mathcal F_{ad}$. Then, we use these conditions to address the following question: let us consider the functional
$$J(\Om_u)=j(u)=\int_0^{2\pi} G\left(\theta,u(\theta),u'(\theta)\right) d\theta$$
where $G:\R\times(0,+\infty)\times\R\to\R$ is $\C^2$, $2\pi$-periodic in the first variable, and $j$ is defined on some set of functions $\mathcal F_{ad}$ as above. Then, is it possible to write down sufficient conditions on $G$ so that any optimal shape of \eqref{eq:Pb2} be a polygon?

This question is motivated by two preliminary results in this direction arising from two quite different fields:
\begin{itemize}
\item first a result by M. Crouzeix in \cite{C05Une}, motivated by abstract operator theory: a problem of the form \eqref{eq:Pb2} is considered where $G(\theta,u,u')=h(u'/u)$ with $h$ strictly concave and even, and $\mathcal F_{ad}=\{u\textrm{ regular enough };\; 0<a\leq u\leq b\}$. All optimal shapes are shown to be polygons,
\item then a result by T. Lachand-Robert and M.A. Peletier for a shape optimization arising in the modeling of the movement of a body inside a fluid (Newton's problem, see \cite{LP01New} and references therein). Here $G(\theta,u,u')=h_1(u)-u'^2h_2(u)$ and $\mathcal F _{ad}=\{u\textrm{ regular enough };\; 0<a\leq u\leq b\}$. Again, with convenient assumptions on $h_1$ and $h_2$, they prove that optimal shapes are polygons.
\end{itemize}
We also refer to T. Bayen \cite{B07Opt} for results about minimizing functionals of type $j$ with similar constraints, seen as controls.

Our goal here is  to generalize these two results and to find rather general sufficient conditions on $G$ which will imply that optimal shapes are necessarily polygons.
We state three results in this direction in the next section. It turns out that a main step in the proof is based on the use of the second order optimality conditions with convexity constraint. This is the main reason why we write down explicitly these conditions, which are actually interesting for themselves and which may also be useful in some other problems (see \cite{HO03Min} for the use of the first order optimality condition on a particular problem of optimal eigenvalue with convexity constraint). They imply that optimal shapes are necessarily polygons inside the constraints (see Theorem \ref{th:concave}). Next, to deal with the solution on the constraint, additional assumptions are needed on the boundary of the constraints (see Theorems \ref{th:polygon-g} and \ref{th:polygon-g-bis}). The sufficient conditions that we obtain on $G$, are rather sharp as shown through several examples and counterexamples.\\

We state sufficient conditions on $G$ for solutions to be polygons in the following section. Then, Section \ref{sect:optimality-conditions} is devoted to the ``abstract'' first and second order optimality conditions for convexity constraint. Proofs of the results in Section \ref{sect:result} are given in Section 4. Finally, we give examples and counterexamples in Section \ref{sect:sharpness} which show how sharp our Section \ref{sect:result} results are.

\section{Main results}\label{sect:result}
{\bf Notation:} $\T:=[0,2\pi)$. Throughout the paper, any function defined on $\T$ is considered as the restriction to $\T$ of a $2\pi$-periodic function on $\R$, with the same regularity.\\
Let $W^{1,\infty}(\T):=\{u\in\;W^{1,\infty}_{loc}(\R)\;/\;u\textrm{ is }2\pi\textrm{-periodic}\}$. If $u\in W^{1,\infty}(\T)$, we say that $u''+u\geq0$ if
\begin{equation}\label{eq:u''+u>=0}
\forall\; v\in W^{1,\infty}(\T)\textrm{ with }v\geq0,\;\;\int_{\T} \left(uv - u'v'\right)d\theta\geq0.
\end{equation}
In this case, $u''+u$ is a nonnegative $2\pi$-periodic measure on $\R$; we then denote $S_u=Supp(u''+u)$ the support of this measure.\\

As explained in the introduction, using the parameterization \eqref{eq:Omega_u}, we consider all open bounded shapes $(\Omega_u)_{u>0}$. A simple calculus of the curvature gives:
$$\kappa(\Om_u)=\frac{u''+u}{\left(1+u'^2\right)^{3/2}},$$
which leads to the characterization of the convexity of $\Om_u$ by the linear condition $u''+u\geq0$. Moreover, straight lines in $\partial\Om_u$ are parameterized by the set $\{u''+u=0\}$, and corners in the boundary are seen as Dirac masses in the measure $u''+u$.\\

\noindent
We consider, as in \cite{LP01New,C05Une}, the geometric constraint $\partial\Om_u\subset A(a,b)$ where $A(a,b):=\{(r,\theta)\;/\;1/b\leq r\leq1/a\}$ is a closed annulus. 
So we consider the problem
\begin{equation}\label{eq:min-g}
\min \left\{ j(u):=J(\Om_u),\quad u\in W^{1,\infty}(\T),\quad u''+u\geq 0,\quad
a\leq u \leq b\right\},
\end{equation}
where $j:W^{1,\infty}(\T)\to\R$, $0<a<b$ are given.
We are interested in sufficient conditions on $j$ (less restrictive as possible) such that the problem
(\ref{eq:min-g}) has for solution a polygon. We also look at the same question for the following problem with the volume constraint $|\Om_u| = m_0$ where $m_0$ is given, namely
\begin{equation}\label{eq:min-g+m}
\min \left\{ j(u),\quad u\in W^{1,\infty}(\mathbb T),\quad u''+u\geq 0,\quad
m(u):=\frac{1}{2}\int_\T \frac{d\theta}{u^2} = m_0
\right\},
\end{equation}
with $m_0>0$. Note that $m(u)$ is the measure of the domain inside the curve $\{(1/u(\theta),\theta),\, \theta\in\T\}$.
\begin{theorem}\label{th:concave}
Let $G:(\theta,u,p)\in\T\times\R\times\R\mapsto G(\theta,u,p)\in\R$ be of class $\C^2$ and
set $j(u)=\int_\T G(\theta,u,u')$. Let $u_0$ be a solution of \eqref{eq:min-g} or \eqref{eq:min-g+m} and assume that $G$ is strongly concave in the third variable at $u_0$, that is to say
\begin{equation}\label{eq:concave}
G_{pp}(\theta,u_0,u_0')<0,\quad \forall\theta\in\mathbb T.
\end{equation}
 \begin{itemize}
       \item If $u_0$ is a solution of \eqref{eq:min-g}, then $S_{u_0}\cap I$ is finite, for any $I=(\gamma_1,\gamma_2)\subset\{\theta\in\R,a<u_0(\theta)<b\}$, and in particular $\Om_{u_0}$ is locally polygonal inside the annulus $A(a,b)$,
       \item If $u_0>0$ is a solution of \eqref{eq:min-g+m}, then $S_{u_0}\cap\T$ is finite, and so $\Om_{u_0}$ is a polygon.
 \end{itemize}
Here $S_{u_0}$ denotes the support of the measure $u_0''+u_0$.
\end{theorem}
See sections \ref{ssect:concave} and \ref{ssect:concave+m} for a proof.  
\begin{remark}\rm
We choose to analyze a volume constraint in \eqref{eq:min-g+m} because this one is classical, and also to show that our approach can be adapted to nonlinear constraints. With a few adjustments, this approach can be adapted to some other constraints, regular enough in terms of $u$, see Proposition \ref{prop:euler+m} and Section \ref{ssect:concave+m}.
\end{remark}
\begin{remark}\rm
 The result is still true if $u_0$ is only a local minimum of \eqref{eq:min-g} or \eqref{eq:min-g+m}, since the proof only use the optimality conditions stated in Section \ref{sect:optimality-conditions}.
\end{remark}

\begin{remark}\rm\label{rk:non-polygon}
 With the only assumptions of Theorem \ref{th:concave}, it is not true that $\Om_{u_0}$ is a polygon if $u_0$ is solution of \eqref{eq:min-g}. Indeed, a solution can saturate the constraint $u\geq a$ or $u\leq b$, and in these cases,  $\partial\Om_{u_0}$ contains an arc of circle. In some particular cases, a solution can also have an infinite number of corners. We refer to Section \ref{sect:sharpness} for explicit examples.
\end{remark}

\noindent
In the following results, we want to go deeper in the analysis, in order to find conditions on $G$ for the solution of \eqref{eq:min-g} to be a polygon. As mentioned in Remark \ref{rk:non-polygon}, we need to avoid that $\partial\Om_{u_0}$ touches the boundary of $A(a,b)$ in an arc of circle, and also an accumulation of corners of $\partial\Om_{u_0}$ in a neighborhood of $\partial A(a,b)$. We treat two kinds of technical assumptions:

\begin{theorem}\label{th:polygon-g}
Let $j(u)=\int_\T G(u,u')$ with $G:(0,\infty)\times \R\to\R$, and let $u_0$ be a solution of (\ref{eq:min-g}). Assume that\\
(i) $G$ is a $\mathcal{C}^2$ function and 
$G_{pp}< 0$ on $\{(u_0(\theta),u_0'(\theta)),\ \theta\in\mathbb T\}$,\\
(ii) The function $p\mapsto G(a,p)$ is even and one of the followings holds \\
\begin{tabular}{ll}
    (ii.1)& $G_u(a,0)< 0$ \textbf{or}\\
    (ii.2)& $G_u(a,0)= 0$ and $G_u(u_0,u_0')u_0 + G_p(u_0,u_0')u_0'\leq0$,\\
\end{tabular}\\
(iii) The function $p\mapsto G(\cdot,p)$ is even and $G_u\geq0$ near $(b,0)$.\\
Then $S_{u_0}$ is finite, i.e. $\Om_{u_0}$ is a polygon.
\end{theorem}
The proof of this theorem follows from Theorem \ref{th:concave} and Proposition \ref{prop:boundary-g}.
\begin{example}\rm
We can give the following geometric example :
$$J(\Om)=\lambda|\Om|-P(\Om),$$
where $|\cdot|$ denotes the area, $P(\cdot)$ denotes the perimeter, and $\lambda\in [0,+\infty]$. The minimization of $J$ within convex sets whose boundary is inside the annulus $A(a,b)$ is in general non trivial.

When $\lambda=0$, the solution is the disk of radius $1/a$ (see \cite{BG97Sha} for a monotony property of perimeter with convex sets). When $\lambda=+\infty$, the solution is the disk of radius $1/b$.

We can easily check (see section \ref{sect:sharpness} for more detailed examples) that $j(u)=J(\Om_u)$ satisfies hypothesis of Theorem \ref{th:concave}, so any solution is locally polygonal inside $A(a,b)$. And from Theorem \ref{th:polygon-g}, if $\lambda\in(a,b)$ (in order to get conditions $(ii)$ and $(iii)$), any solution is a polygon.
\end{example}

We can prove the same result as in Theorem \ref{th:polygon-g} with a weaker condition than the
uniform condition given in (i), namely when $G_{pp}(a,p)=0$, like in \cite{LP01New}.
\begin{theorem}\label{th:polygon-g-bis}
Let $j(u)=\int_\T G(u,u')$ with $G:(0,\infty)\times \R\to\R$, $C(b)=2\pi b$ (see Lemma \ref{lem:borne-u'}) and let $u_0$ be a solution of (\ref{eq:min-g}).
 We assume that\\
(i) $G$ is a $\mathcal{C}^3$ function, 
$G_{pp}=0$ in $\{a\}\times[-C(b),C(b)]$, and $G_{pp}<0$ in $(a,b]\times[-C(b),C(b)]$,\\
(ii) $p\to G(a,p)$ is even, $G_u(a,p)<0$ for all $p\in[-C(b),C(b)]$ and $pG_{up}(a,p)=z(p)G_{upp}(a,p)$ for $p\in(0,C(b)]$, with a certain function $z\geq 0$,\\
(iii) $p\to G(\cdot,p)$ is even and $G_u\geq0$ near $(b,0)$.\\
Then $S_{u_0}$ is finite, i.e. $\Om_{u_0}$ represents a polygon.
\end{theorem}
The proof of this theorem follows from Propositions \ref{prop:concave-bis} and \ref{prop:boundary-g-bis}.

\begin{remark}\rm
The hypotheses in Theorem \ref{th:polygon-g} and \ref{th:polygon-g-bis} are quite general.
In Section \ref{sect:sharpness} we give certain examples showing that if one of these hypotheses is not satisfied, then the solutions of (\ref{eq:min-g}), in general,
are not polygons.
\end{remark}
\begin{remark}\rm\label{rk:MC}
The condition (ii.2) in Theorem \ref{th:polygon-g} (less natural than (ii.1)) has been motivated by the problem in \cite{C05Une}, where $G(u,p)=h(p/u)$ with $h(\cdot)$ a $\mathcal{C}^2$, strictly concave, and  even function. Such a $G(u,p)$ satisfies the hypothesis of Theorem \ref{th:polygon-g}. Indeed,\\
\begin{tabular}{lp{14cm}}
(a) & $G_{pp}(u,p)=h''(p/u)u^{-2}$, so  $G_{pp}(u,p)<0$ and (i) is satisfied,\\
(b) & $G_u(u,p)u+G_p(u,p)p=0$ and $G_u(\cdot,0)=0$, so (ii.2) is satisfied,\\
(c) & $G_u(u,p)=-h'(p/u)\frac{p}{u^2}\geq 0$ so (iii) is satisfied.\\
\end{tabular}\\
Therefore the solution is a polygon. In \cite{C05Une}, several more precise statements about the geometric nature of solutions are proven (in this particular case).
\end{remark}
\begin{remark}\rm\label{rk:LR}
Similarly, Theorem \ref{th:polygon-g-bis} gives a generalization of the problem studied in \cite{LP01New}. Indeed, in this problem,
they have $G(u,p)=h_1(u)-p^2 h_2(u)$ with $h_1,h_2$ two $\mathcal{C}^2$ functions satisfying $h_1'(a)<0$, $h_1'(b)>0$, $h_2(a)=0$, and
$\forall\; t>a,\; h_2(t)>0$ ($G$ is not $\mathcal{C}^3$ in this case, but in fact we only need the existence of $G_{upp}$, which is clear here).\\
The function $G(u,p)$ satisfies the hypothesis
of Theorem \ref{th:polygon-g-bis} as $p\to G(u,p)$ is even and\\
\begin{tabular}{lp{14cm}}
(a) & $G_{pp}(u,p)=-2h_2(u)$, so (i) is satisfied.\\
(b) & $G_u(a,p)=h_1'(a)<0$ and $G_{up}(a,p)=-2p h_2'(a)$, $G_{upp}(a,p)=-2h_2'(u)$, so $G_{up}(u,p)=pG_{upp}(u,p)$, and therefore (ii) is satisfied.\\
(c) & $G_u(u,p)=h_1'(u)-p^2h_2'(u)$ so $G_u(b,0)=h_1'(b)>0$. 
\end{tabular}\\
This last assumption is not specified in \cite{LP01New}, but according to us, we need this one, see Section \ref{ssect:sharpness}. In fact, it seems that the case of an accumulation of corners in the interior boundary $\{u_0=b\}$ is not considered in \cite{LP01New} (see Proposition \ref{prop:boundary-g}, case (b)).\\
So the solution is a polygon. In \cite{LP01New}, it is also proven that this polygon is regular in this particular case.
\end{remark}

\begin{remark}\rm\label{rk:exist}
Les us make some comments on the question of existence. For the problem (\ref{eq:min-g}), there always exists a solution, if for example $j$ is continuous in $H^1(\T)$ (see below for a definition). Indeed, the minimization set $\{u\in W^{1,\infty}(\T)\;/\;u''+u\geq0, a\leq u\leq b\}$ is strongly compact in $H^1(\T)$.\\
About the problem (\ref{eq:min-g+m}) with a measure constraint, the question is more specific. For example, if one looks at the problem of maximization of the perimeter (for which the concavity assumptions is satisfied), with convexity and measure constraints, we are in a case of non-existence (the sequence of rectangles $\Omega_n=(-n/2,n/2)\times(-m_0/2n,m_0/2n)$ satisfies the constraints, whereas the perimeter is going to $+\infty$). However, existence may be proved for many further functionals. In Theorem \ref{th:concave}, we avoid this issue by asking the solution to be positive (and so to represent a convex bounded set of dimension 2).
\end{remark}

\section{First and second order optimality conditions}\label{sect:optimality-conditions}

As we noticed in Remark \ref{rk:exist}, the minimization set is compact. So there are
very few directions to write optimality. However, we are able in this section to
write general optimality conditions for our problem.\\

Let us first introduce an abstract setting (see \cite{IT79The}, \cite{MZ79Fir}). Let $U, Y$ be two real Banach spaces, let $K$ be a nonempty closed convex cone in $Y$ and let $f:U\to \R, \;g:U\to Y$. We consider the minimization problem
\begin{eqnarray}\label{abstract-pb}
\min\{ f(u),\ u\in U,\;g(u)\in K\}.
\end{eqnarray}
We denote by $U'$ (resp. $Y'$) the Banach space of continuous linear maps from $U$ (resp. $Y$) into $\R$ (dual spaces of $U,Y$), and we introduce
$$Y'_+=\{l\in Y';\;\forall \,k\in K,\;  l(k)\geq 0\;\}.$$
The following result is a particular case of Theorem 3.2 and 3.3 stated in \cite{MZ79Fir} which will be sufficient for our purpose.
\begin{prop} \label{main}
Let $u_0\in U$ be a solution of the minimization problem (\ref{abstract-pb}). Assume $f$ and $g$ are twice (Fr\'echet-)differentiable at $u_0$ and that $g'(u_0)(U)=Y$. Then,
\begin{enumerate}
\item[(i)] 
there exists $l\in Y'_+$ such that $f'(u_0)=l\circ g'(u_0)$ and $l(g(u_0))=0$,
\item[(ii)] if $F(u):=f(u)-l(g(u))$, then $F''(u_0)(v,v)\geq 0$ for all $v\in T_{u_0}$ where
$$T_{u_0}=\big\{v\in U;\;f'(u_0)(v)=0,\; g'(u_0)(v)\in K_{g(u_0)}=\{K+\lambda g(u_0);\lambda\in\R\}\big\}.$$
\end{enumerate}
\end{prop}
\begin{remark}\rm\label{Tu0} When applying the second order optimality condition (ii), we have to check whether well-chosen $v\in U$ are in $T_{u_0}$. This may be done by using (i) and the information on the linear map $l$. We may use instead the following: assume $g(u_0+tv)\in K$ for $t>0$ small, or, more generally that
\begin{equation}\label{T_u}
 u_0+t v=v_t+t\eps(t)\textrm{ with }\lim_{t\to 0,t>0}\eps(t)=0\textrm{ and }\;g(v_t)\in K;
\end{equation}
then
\begin{equation}\label{Euler-Inequality}
 f'(u_0)(v)\geq 0\textrm{ and } g'(u_0)(v)\in K_{g(u_0)}.
\end{equation}
To see this, we write the two following lines:
$$0\leq t^{-1}[f(v_t)-f(u_0)]=f'(u_0)(v)+\eps_1(t)\;where\;\lim_{t\to 0,t>0}\eps_1(t)=0,$$
$$g'(u_0)(v)=t^{-1}[g(v_t)-g(u_0)]+\eps_2(t)\;where\;\lim_{t\to 0,t>0}\eps_2(t)=0,$$
and we let $t$ tend to zero.\\
If now, \eqref{T_u} is valid for all $t$ small ($t>0$ and $t<0$), then $v\in T_{u_0}$.
\hfill$\Box$
\end{remark}

For our purpose, we choose $U=H^1(\T)$ the Hilbert space of functions from $\R$ into $\R$ which are in $H^1_{loc}(\R)$ and $2\pi$-periodic, equipped with the scalar product
$$\forall u,v\in U,\;(u,v)_{U\times U}=\int_\T u\,v+u'v'.$$

Let $g_0:U\to U'$ be defined by
$$\forall u,v\in U,\; g_0(u)(v)=\int_\T u\,v-u'v'.$$
For $l\in U'$ we say $l\geq0$ in $U'$ if $l(v)\geq0$ for all $v\in U$. Note that,  if $g_0(u)\geq 0$ in $U'$ then $u+u''$, computed in the sense of distributions in $\R$ , is a $2\pi$-periodic nonnegative measure on $\R$, and we have
\begin{eqnarray}\label{u+u''} g_0(u)(v)=\int_\T u\,v-u'v'=\int_\T v\,d(u+u'').
\end{eqnarray}
Note also, for further purposes, that $g_0(U)$ is a closed subspace of $U'$ which may be described as the "orthogonal" of the kernel of $g_0$ (because $\overline{R(g_0)}=N(g_0^*)^{\perp}$, with $g_0^*$ the adjoint of $g_0$), namely
$$g_0(U)=\{z\in U';\;\forall v\in Ker\, g_0,\;z(v)=0\}=\{z\in U';\;z(\cos)=z(\sin)=0\},$$
(and $\cos,\sin$ denote the usual cosine and sine functions on $\R$).

Finally, if $l$ is a continuous linear map from $g_0(U)$ into $\R$ (that is $l\in g_0(U)'$), then, thanks to the Hilbert space structure, there exists $\zeta\in U$ such that
\begin{eqnarray}\label{zeta}
\forall z\in g_0(U),\;l(z)=\langle z,\zeta\rangle_{U'\times U},\; (\zeta,cos)_{U\times U}=(\zeta,sin)_{U\times U}=0.
\end{eqnarray}


\noindent{\bf {\large First problem:}}\\
Let $j:U\to \R$ be $\C^2$. We set $Y:=g_0(U)\times U\times U$ equipped with its canonical Hilbert space structure whose scalar product writes: $\forall y=(z,u_1,u_2), \widehat{y}=(\widehat{z},\widehat{u_1},\widehat{u_2})\in Y$,
$$\langle y,\widehat{y}\rangle_{Y\times Y}:=\langle z,\widehat{z}\rangle_{U'\times U'}+ (u_1,\widehat{u_1})_{U\times U}+(u_2,\widehat{u_2})_{U\times U}.$$
And we define $g:U\to Y$ and $K\subset Y$ by 
$$g(u)=(g_0(u),u-a,b-u),\;K=\{(z,u_1,u_2)\in Y;\;z\geq 0\;in\;U',\,u_1,u_2\geq 0\;in\;U\}.$$

We look at the minimization problem
(see Lemma \ref{lem:borne-u'} and Remarks \ref{rem:spaces} and \ref{r:tildeG} for details about the choice of the two functional spaces $H^1(\T)$ and $W^{1,\infty}(\T)$):
\begin{eqnarray}\label{opt1}
\min\{ j(u),\; u\in U,\; g(u)\in K\}.
\end{eqnarray}

\begin{prop}\label{prop:euler}
If $u_0$ is a solution of (\ref{opt1}) where $j:H^1(\T)\to\R$ is $\mathcal{C}^2$, then there exist $\zeta_0 \in H^1(\T)$ nonnegative, $\mu_a, \mu_b \in \mathcal{M}^+(\T)$ (space of nonnegative Radon measure on
 $\T$) such that 
\begin{eqnarray}\label{zeta_0}  
\zeta_0=0\;on\;S_{u_0},\quad Supp(\mu_a)\subset\{u_0=a\},\quad Supp(\mu_b)\subset\{u_0=b\}\\
\mbox{ and }\forall\; v\in H^1(\T),\;
j'(u_0)v=\displaystyle{
\langle\zeta_0+\zeta_0'',v\rangle_{U'\times U}+\int_{\T}vd\mu_a-\int_{\T}vd\mu_b}.\label{eq:euler1}
\end{eqnarray}
Moreover, $\forall\; v\in H^1(\T) \mbox{ such that }\exists\lambda\in\R \mbox{ with }
\left\{\begin{array}{l}
v''+v\geq \lambda(u_0''+u_0)\\
v\geq\lambda(u_0-a),v\leq\lambda(u_0-b),\\
\langle\zeta_0+\zeta_0'',v\rangle_{U'\times U}+\int_\T vd(\mu_a-\mu_b)=0
\end{array}\right.$
\begin{equation}\label{eq:euler2}
\mbox{ we have }\;\;\;j''(u_0)(v,v)\geq 0.
\end{equation}
\end{prop}
\begin{remark}\rm\label{rem:spaces}
We choose here to work in the space $H^1(\T)$, whereas the problem is more naturally settled in $W^{1,\infty}(\T)$. This choice is motivated by the fact that $H^1(\T)$ is reflexive. If we had worked in $W^{1,\infty}(\T)$, we would have obtained a Lagrange multiplier in the bidual $W^{1,\infty}(\T)''$, which is not so easy to make explicit. Nevertheless, this choice of $H^1(\T)$ leads to this new difficulty: for $G$ regular, the functional $j(u)=\int_\T G(\theta,u,u')$ is generally not well defined on $H^1(\T)$, and so we cannot directly apply Proposition \ref{prop:euler}. We explain in Section \ref{ssect:concave} the adjustments that are needed to apply this one.
\end{remark}
\begin{proof}
We apply Proposition \ref{main} with the notations just introduced above. The main assumption $g'(u_0)(U)=Y$ is satisfied since $g'(u_0)=(g_0,I,-I)$ where $I$ denotes the identity. By the statement (i) there exists $l=(l_0,l_a,l_b)\in Y'_+$ and thanks to the remarks (\ref{u+u''}), (\ref{zeta}), there exists $(\zeta_0,\mu_a,\mu_b)\in U\times U'\times U'$ such that\\
\begin{itemize}
\item $\forall v\in U, j'(u_0)(v)=\langle\zeta_0+\zeta''_0,v\rangle_{U'\times U}+\int_\T vd\mu_a-\int_\T vd\mu_b$,
\item $\mu_a$ is a nonnegative measure and $\int_\T (u-a)d\mu_a=0$ or $Supp(\mu_a)\subset\{u_0=a\}$,
\item  $\mu_b$ is a nonnegative measure and $\int_\T (b-u)d\mu_b=0$ or $Supp(\mu_b)\subset\{u_0=b\}$,
\item  $(\zeta_0,\cos)_{U\times U}=(\zeta_0,\sin)_{U\times U}=0,\;$  $\int_\T \zeta_0d(u_0+u_0'')=0$ and
\begin{equation}\label{zeta_0<0}
\forall v\in U\;with\;g_0(v)\geq 0,\; \int_\T \zeta_0v-\zeta_0'v'\geq 0.
\end{equation}
\end{itemize}
Let now $v\in U$ with $v+v''=\psi(\theta)(u_0+u_0'')$ with $\psi$ Borel measurable and bounded. Then, 
$g_0(\|\psi\|_{L^\infty} u_0\pm v)\geq 0$ so that 
$\langle\zeta_0, g_0(\|\psi\|_{L^\infty} u_0\pm v)\rangle_{U\times U'}\geq 0$. It follows that 
\begin{equation}\label{1}
|\langle\zeta_0,g_0(v)\rangle_{U\times U'}|\leq \|\psi\|_{L^\infty} \langle\zeta_0, g_0(u_0)\rangle_{U\times U'}=0.
\end{equation}
But this information on $\zeta_0$ is not sufficient to obtain the first property of \eqref{zeta_0}, namely $\zeta_0(u_0''+u_0)\equiv0$. For this, we now show that it is possible to change $\zeta_0$ into $Z_0=\zeta_0+a \cos  + b \sin$ so that all same properties remain valid, but also $Z_0(u_0''+u_0)\equiv0$.

Since $\int_\T(v+v'')\cos=0=\int_\T(v+v'')\sin$, we also have \eqref{eq:euler1} for $Z_0$ in place of $\zeta_0$. Moreover, \eqref{1} is also true for $Z_0$, that is to say: for every  $\psi$ Borel measurable and bounded such that $v+v''=\psi(\theta)(u_0+u_0'')$ for some $v\in U$, $\int_\T\psi Z_0d\nu=0$, where we denote $\nu=u_0+u_0''$. 

Let us show that we can find $a,b\in\R$ so that $\;Z_0\nu\equiv 0\;$ and $Z_0\geq 0$, and the proof of Proposition \ref{prop:euler} will be complete.\\
Let us choose $a,b$ so that 
\begin{eqnarray}\label{oh}
\int_\T\cos \theta Z_0(\theta)d\nu(\theta)=0=\int_\T\sin \theta Z_0(\theta)d\nu(\theta),
\end{eqnarray}
which writes
\begin{eqnarray}
\left\{
\begin{array}{l}
\int_\T\cos \theta \zeta_0(\theta) d\nu(\theta)+a\int_\T\cos^2 \theta\,d\nu(\theta)+b\int_\T\cos \theta \sin\theta d\nu(\theta)=0,\\
\int_\T\sin \theta \zeta_0(\theta) d\nu(\theta)+a\int_\T\sin \theta\cos\theta\,d\nu(\theta)+b\int_\T\sin^2\theta d\nu(\theta)=0.
\end{array}
\right.
\end{eqnarray}
This is possible since $CS-B^2\neq 0$ where
$$C=\int_\T\cos^2\theta \,d\nu(\theta),\;S=\int_\T\sin^2\theta \,d\nu(\theta),\;B=\int_\T\cos \theta\sin\theta\, d\nu(\theta).$$
Indeed, by Schwarz' inequality, we have $B^2\leq CS$ and equality would hold only if we had
$$ \cos \theta=\lambda \sin \theta\;\; \nu-a.e\; \theta,$$
for some $\lambda\in \R$  and it is not the case since $\nu$ has at least 3 distinct points in its support.\\
Let now $\varphi$ be a Borel measurable bounded function on $\T$. Let $\psi:=\varphi+\alpha \cos+\beta\sin$ where $\alpha,\beta\in\R$ are chosen so that $\psi\nu=v+v''$ for some $v\in U$, or equivalently
\begin{eqnarray}
\left\{
\begin{array}{l}
\int_\T\cos \theta [\varphi+\alpha\cos+\beta \sin](\theta)d\nu(\theta)=0,\\
\int_\T\sin \theta [\varphi+\alpha\cos+\beta \sin](\theta)d\nu(\theta)=0.
\end{array}
\right.
\end{eqnarray}
Again, this is possible since $CS-B^2\neq 0$. Next, we deduce from (\ref{oh}), then from (\ref{1}) that
$$\int_\T\varphi Z_0d\nu=\int_\T\psi Z_0 d\nu=0.$$
By arbitrarity of $\varphi$, this implies $Z_0\nu\equiv 0$ as expected. This gives \eqref{zeta_0} and \eqref{eq:euler1} with $Z_0$ in place of $\zeta_0$.\\
We now prove that $Z_0$ is nonnegative : $Supp(u_0''+u_0)^c=\bigcup_n\omega_n$ where $\omega_n$ are open intervals. Then, if $\psi\geq 0$ is regular with a compact support in $\omega_n$, we can introduce $v\in H^1_0(\omega_n)$ satisfying $v''+v=\psi$ in $\omega_n$ (possible since $diam(\omega_n)<\pi$). We define $v$ by 0 outside $\omega_n$. Thus $v''+v$ has Dirac mass at $\partial\omega_n$, but since $Z_0$ vanishes at $\partial\omega_n$, we finally get, using \eqref{zeta_0<0}:
$$\int_{\omega_n} Z_0\psi d\theta=\int_{\T} Z_0d(v''+v) = 0.$$
Since $\psi$ is arbitrary, we get $Z_0\geq 0$ in $\omega_n$, and then $Z_0\geq 0$ in $\T$.

By the statement (ii) of Proposition \ref{main}, for each $v\in U$ satisfying
\begin{eqnarray}\label{tangent}
f'(u_0)(v)=0, \;\exists \lambda \in\R, \;g_0(v)\geq \lambda g_0(u_0),\; v\geq \lambda(u_0-a),\;v\leq \lambda (u_0-b),
\end{eqnarray}
we have $f''(u_0)(v,v)\geq 0$ (the constraint $g$ is linear, so $g''=0$). Whence Proposition \ref{prop:euler}, with $Z_0$ in place of $\zeta_0$.\qed
\end{proof}
\begin{remark}\rm
In general, the positivity of $\zeta_0$ on the orthogonal of $\{cos,sin\}$ does not imply that it is pointwise positive (one can write explicit examples).
\end{remark}
\begin{remark}\rm
In the following section, the main difficulty will be to analyze the situation where the convexity constraint is almost everywhere saturated. It would be easy to prove the non-existence of an nonempty interval $I\subset S_{u_0}\cap\{a<u_0<b\}$. However, this is not sufficient to conclude that $u_0''+u_0$ is a sum of Dirac masses (we can look at the Lebesgue decomposition of measures to see this).
That is why we have to analyze the case of infinitely many corners, or even of a diffuse singular measure (see the proof of Theorem \ref{th:concave}).\\
Another way to avoid these difficulties has been chosen by M. Crouzeix in \cite{C05Une} for his particular problem (see Remark \ref{rk:MC}): he considers the minimization problem
restricted to convex polygons having at most $n$ edges, and proves that with $n$ large enough, any
solution in this restricted class has only $n_0$ edges where $n_0$ is only determined by $a$ and $b$. Therefore, using the density of convex polygons in convex sets, the solution for this particular problem (\ref{eq:min-g}) is still a polygon.
\end{remark}
\begin{remark}\rm
Our analysis in Section \ref{sect:proofs} could easily show on some simple examples that the first order equation is not sufficient to get the results of Theorems \ref{th:concave}, \ref{th:polygon-g} or \ref{th:polygon-g-bis}. It turns out that the second order condition is very helpful.
\end{remark}

\noindent {\bf {\large Second problem:}}\\
Similarly, we can give the optimality conditions in the case of the measure constraint:
\begin{equation}\label{eq:min-g+m0}
\min \left\{ j(u),\quad u\in H^1(\mathbb T),\quad u''+u\geq 0,\quad
m(u):=\frac{1}{2}\int_\T \frac{d\theta}{u^2} = m_0
\right\},
\end{equation}
\begin{prop}\label{prop:euler+m}
If $u_0$ solves (\ref{eq:min-g+m0}) where $j:H^1(\T)\to\R$ is $\C^2$, then there exist $\zeta_0 \in H^1(\T)$ nonnegative, $\mu\in\R$ such that 
\begin{eqnarray}\label{zeta_0-bis}
&&\zeta_0=0\;on\;S_{u_0},\\
\mbox{ and }&\forall\; v\in H^1(\T),&
j'(u_0)v =\displaystyle{
\langle\zeta_0+\zeta_0'',v\rangle_{U'\times U}-\mu m'(u_0)(v)}.\label{euler1-ter}
\end{eqnarray}
Moreover, for all $v\in H^1(\T)$, such that $\exists\lambda\in\R$ satisfying 
$\left\{\begin{array}{l}
v''+v\geq \lambda(u_0''+u_0)\\
\langle\zeta_0+\zeta_0'',v\rangle_{U'\times U}-\mu m'(u_0)(v)=0
\end{array}\right.$,
\begin{equation}\label{H-g+m>0}
\mbox{ we have }j''(u_0)(v,v)+\mu m''(u_0)(v,v)\geq 0.
\end{equation}
\end{prop}
\begin{proof}
We make the same choices except for
$$Y=g_0(U)\times\R,\;g:U\to Y, \forall u\in U,\;g(u)=(g_0(u),m(u)-m_0),$$
and $K= \{z \in g_0(U),\; z\geq 0\;in\;U'\}\times \{0\}\subset Y$.

Here, using $min(u_0)>0$, we have
$$\forall v\in U, g'(u_0)(v)=(g_0(v),m'(u_0)(v))=\left(g_0(v),-\int_\T \frac{v\,d\theta}{u_0^3}\right)$$ 
and $g'(u_0)(U)=g_0(U)\times  \R=Y$. Therefore, we may apply Proposition \ref{main}, and similarly to the proof of Proposition \ref{prop:euler}, we get the result.\qed
\end{proof}

\section{Proofs}\label{sect:proofs}
\subsection{Proof of Theorem \ref{th:concave}, case of inclusion in $A(a,b)$}\label{ssect:concave}
First of all, we have to prove that $u'$ is bounded by a constant $C(b)$, for all $u$ admissible.
\begin{lemma}\label{lem:borne-u'}
$$\forall u\in H^1(\T),\;\Big[\;0\leq u\leq b, u''+u\geq 0\; \Longrightarrow\; \|u'\|_{L^\infty} \leq 2\pi b=:C(b)\; \Big]$$
\end{lemma}
\textbf{Proof of lemma \ref{lem:borne-u'}}
Since $u$ is periodic, there exists $x_0\in\T$ such that $u'(x_0)\geq 0$. With $x\in[x_0,x_0+2\pi]$ and integrating the inequality $u''+u\geq 0$, we get
$u'(x)-u'(x_0)+\int_{x_0}^x u \geq 0$
which leads to $u'(x)\geq -2\pi b$, true for all $x\in\R$ by periodicity.\\
Similarly with $x_1$ such that $u'(x_1)\leq 0$ and $x\in[x_1-2\pi,x_1]$, we get $u'(x)\leq 2\pi b$
which leads to the result with $C(b)=2\pi b$.\qed

\begin{remark}\rm\label{r:tildeG}
With the help of this lemma, let us explain how we can use Proposition \ref{prop:euler}, whereas $j(u)=\int_\T G(\theta,u,u')$ is a priori not defined on $H^1(\T)$:
if $\eta(u,p)$ is a $\mathcal C^\infty$ cut-off function, with $0\leq\eta\leq 1$ and such that
\[
\eta=\left\{
\begin{array}{lll}
1,& (u,p)\in [a/2,2b]\times[-2C(b),2C(b)],\\
0,& \mbox{otherwise},
\end{array}
\right.
\]
where $C(b)$ is introduced in Lemma \ref{lem:borne-u'}, then we can set $\widetilde{j}(u):=\int_{\mathbb T}\widetilde{G}(\theta,u,u')d\theta$, with $\widetilde{G}(\theta,u,p):= \eta(u,p)G(\theta,u,p)$. Easily, the new functional $\widetilde{j}$ is $\mathcal C^k$ in $H^1(\mathbb T)$ if $G$ is $\mathcal C^k$ in $\T\times\R\times\R$. Moreover, by the choice of $\eta$, any solution of the problem (\ref{eq:min-g}) is still solution for $\widetilde{j}$ instead of $j$, and we can write first and second order necessary conditions for the function $\widetilde{j}$, in terms of $\widetilde{G}$.\\
We easily check that $\widetilde{G}$ still satisfy the hypothesis in Theorem \ref{th:concave}, since $\eta=1$ in a neighborhood of $[a,b]\times[-C(b),C(b)]$ (this will also be true for Theorems \ref{th:polygon-g} and \ref{th:polygon-g-bis}). We drop the notation $\widetilde{\cdot}$ in all what follows.
\end{remark}

\noindent\textbf{Proof of Theorem \ref{th:concave}, case of inclusion in $A(a,b)$:}\\
Assume by contradiction that $u_0$ does not satisfy the conclusion.
Therefore there exists an interval $I\subset\{a<u_0<b\}$ and $\theta_0$ an accumulation point  of
$S_{u_0}\cap I$.\\
~\\
\begin{minipage}{95mm}
\textbf{(a) Case $a<u_0(\theta_0)<b$.}\\
Without loss of generality we can assume
$\theta_0=0$ and also that there exists a decreasing sequence $(\eps_n)$ tending to $0$ such that
$S_{u_0}\cap(0,\eps_n)\neq \emptyset$. Then we follow an idea of T. Lachand-Robert
and M.A. Peletier (see \cite{LP01New}). We can always find $0<\eps_n^i<\eps_n$, $i=1,\ldots,4$, increasing with respect to $i$, such that $S_{u_0}\cap(\eps_n^i,\eps_n^{i+1})\neq\emptyset$, $i=1,3$. We consider $v_{n,i}$ solving
\[
 v_{n,i}''+v_{n,i}=\chi_{(\eps_n^i,\eps_n^{i+1})}(u_0''+u_0),\quad v_{n,i}=0 \mbox{ in }
(0,\eps_n)^c,\; i=1,3.
\]
Such $v_{n,i}$ exist since we avoid the spectrum of the Laplace operator with Dirichlet boundary conditions.
Next, we look for $\lambda_{n,i},\ i=1,3$ such that
${\displaystyle v_n=\sum_{i=1,3}\lambda_{n,i} v_{n,i}}$ satisfy
\[
v'_n(0^+)=v'_n(\eps_n^-)=0.
\]
\end{minipage}
\begin{minipage}{70mm}
\figinit{.9pt}
\figpt 0:(0,0)
\figpt 1:(0,100)
\figpt 2:(-20,0)

\figpt 3:(0,85)
\figpt 4:(31,79)
\figpt 5:(50,60)
\figpt 6:(55,40)
\figpt 7:(54,30)
\figpt 8:(50,20)
\figpt 9:(47,15)
\figpt 10:(45,13)

\figpt 11:(0,-30)
\figpt 12:(-50,-20)
\figpt 13:(-50,50)
\figpt 14:(-30,70)
\figpt 15:(40,-30)

\psbeginfig{}
\psset(width=0.3)

\psarccirc 0;100(-20,130)
\pscirc 0(20)

\psset(width=0.8)
\psarccirc 0 ; 85(68.5,90)
\psline[4,5,6,7,8,9]
\psline[10,11,12,13,14,3]
\psset(dash=8)
\psline[9,10]

\psset(width=0.2)
\psline[0,3]
\psline[0,4]
\psline[0,5]
\psline[0,6]
\psline[0,7]
\psline[0,8]
\psline[0,9]
\psline[0,10]

\psendfig
\figvisu{\figBoxA}{\bf Figure 1: Case (a)}{
\figwriten 1:{$\frac{1}{a}$}(2)
\figwritew 2:{$\frac{1}{b}$}(2)
\figwritene 5:{$\eps_n$}(2)
\figwritew 12:{$\Om_{u_0}$}(2)
\figwritee 7:{$\eps_n^i$}(4)
\figwritee 10:{$\theta_0$}(4)
\figwrite [15]{$A(a,b)$}
}
\centerline{\box\figBoxA}
~\\
\end{minipage}
\\
The above derivatives exist since $v_{n,i}$ are regular near $0$ and $\eps_n$ in $(0,\eps_n)$.
We can always find such $\lambda_{n,i}$ as they satisfy two linear equations.
It implies that $v_n''$ does not have any Dirac mass at $0$ and $\eps_n$.
Since $S_{u_0}\cap(\eps_{n}^i,\eps_{n}^{i+1})\neq \emptyset$, we have $v_n\neq 0$. From \eqref{zeta_0} and $Supp(v_n)\subset \{a<u_0<b\}$ it follows that for such $v_n$ we have
\[
\int_{\T}v_n(\zeta_0+\zeta_0'')=
\int_{\T} v_n d\mu_a=\int_{\T} v_n d\mu_b=0.
\]
Using the first order Euler-Lagrange equation (\ref{eq:euler1}), we get $j'(u_0)(v_n)=0$.
Consequently, $v_n$ is eligible for the second order necessary condition (it is easy to check the other conditions required in Proposition \ref{prop:euler}).  So, using (\ref{eq:euler2}), we get
\begin{equation*}
0\leq
j''(u_0)(v_n,v_n)=
\int_{\T}
G_{uu}(\theta,u_0,u_0')v_n^2 + 
2 G_{up}(\theta,u_0,u_0')v_n v_n' + G_{pp}(\theta,u_0,u_0'){v'_n}^2.
\end{equation*}
Using the concavity assumptions \eqref{eq:concave} on $G$, it follows that
\begin{eqnarray}\label{eq:main}
0&\leq&
j''(u_0)(v_n,v_n)
\leq
\int_{\T}
K_{uu}v_n^2 + 2K_{up}|v_n||v_n'| - K_{pp}|v_n'|^2 \nonumber\\\label{eq:eps_n}
&\leq&
\left(
    \left(\frac{\eps_n}{\pi}\right)^2K_{uu}+2\frac{\eps_n}{\pi}K_{up} -
    K_{pp}\right)\|v_n'\|^2_{L^2},
\end{eqnarray}
where, if we set $R:=\T\times[a,b]\times[-C(b),C(b)]$, we have
\begin{equation}\label{eq:K_}
K_{uu}=\sup_R |G_{uu}|,\quad 
K_{up}=\sup_R |G_{up}|,\quad
K_{pp}=\inf_{\T} |G_{pp}(\theta,u_0(\theta),u_0'(\theta))|>0.
\end{equation}
In order to get \eqref{eq:eps_n}, we have used Poincar\'e's inequality $\forall\; v\in H^1_0(0,\eps),\;\int_0^\eps u^2 \leq \left(\frac{\eps}{\pi}\right)^2\int_0^\eps u'^2$, with
$\eps=\eps_n$.
As $\eps_n$ tends to $0$, the inequality (\ref{eq:eps_n}) becomes impossible and
proves that $S_{u_0}$ has not interior accumulation points.
It follows that $u_0''+u_0$ is a sum of positive Dirac masses,
$u_0''+u_0=\sum_{n\in\N}\alpha_n \delta_{\theta_n}$ in $\{a<u_0<b\}$.\\
%
~\\
\begin{minipage}{100mm}
\textbf{(b) Case $u_0(\theta_0)=a$.} From (a), it follows that near $\theta_0$ and at least from one side of it we have 
$u_0''+u_0=\sum_{n\in \N^*}\alpha_n\delta_{\theta_n}$ where $\{\theta_n\}$ is a sequence such that $\theta_n\to\theta_0$, $\theta_n\in S_{u_0}\cap I$ and $\alpha_n>0$. Without restriction, we may take $\theta_0=0$ and assume that $\theta_n>0$ is decreasing. For every $n$ we consider
$v_n\in H_0^1(\theta_{n+1},\theta_{n-1})$ satisfying
$v_n''+v_n=\delta_{\theta_n}$ in $(\theta_{n+1},\theta_{n-1})$.
In $\T$, the measure $v_n''+v_n$ is supported in $\{\theta_{n+1},\theta_n,\theta_{n-1}\}$, and since these points are in $S_{u_0}$, and since $u_0$ does not touch $a$ in a neighborhood of $[\theta_{n+1},\theta_{n-1}]$, we can choose $\lambda\ll0$ (depending on $n$) such that 
$$
\left\{\begin{array}{l}
v_n''+v_n\geq \lambda(u_0''+u_0)\\
v_n\geq\lambda(u_0-a),v_n\leq\lambda(u_0-b).
\end{array}\right.$$
\end{minipage}
\begin{minipage}{60mm}
\figinit{1.1pt}
\figpt 0:(0,0)
\figpt 1:(-20,100)
\figpt 2:(-5,20)

\figpt 3:(0,85)
\figpt 4:(40,79)
\figpt 5:(65,65)
\figpt 6:(83,50)
\figpt 7:(90,40)
\figpt 8:(92.5,35)
\figpt 9:(94,32)
\figpt 10:(95,30)


\psbeginfig{}

\psarccirc 0; 100(0,100)
\psarccirc 0; 20(0,100)

\psset(width=0.8)
\psline[3,4,5,6,7,8,9,10]
\psset(dash=8)
\psset(width=0.2)
\psline[0,3]
\psline[0,4]
\psline[0,5]
\psline[0,6]
\psline[0,7]
\psline[0,8]
\psline[0,9]
\psline[0,10]

\psendfig
\figvisu{\figBoxA}{\bf Figure 2: Case (b)}{
\figwritew 1:{$\frac{1}{a}$}(2)
\figwritew 2:{$\frac{1}{b}$}(2)
\figwritee 7:{$\theta_n$}(3)
\figwritee 10:{$\theta_0$}(2)
\figwritew 3:{$\partial\Omega_{u_0}$}(2)
}
\centerline{\box\figBoxA}
\end{minipage}
~\\
Moreover, 
 since $v_n$ is supported in $\{a<u_0<b\}$, we finally get, using \eqref{zeta_0}, $\int_\T vd(\zeta_0+\zeta_0''+\mu_a-\mu_b)=0$,
and so the function $v_n$ is admissible for the second order necessary condition. Proceeding as
in (a) above, we find a contradiction which proves that this case is impossible. \\
~\\
\textbf{(c) Case $u_0(\theta_0)=b$.} This case is
treated similarly to the case (b).\qed
\begin{cor}\label{cor:borne-u'}
We have $\|u_0'\|_{L^\infty}\leq \sqrt{2b(b-a)}$. More generally,
if $u\in H^1(\mathbb T)$, $0<\alpha\leq u\leq \beta<\infty$ and 
$u''+u\geq0$ with $|\{\alpha <u <\beta\}\cap{\rm Supp}(u''+u)|=0$ then
$$
\|u'\|_{L^\infty} \leq \sqrt{2\beta(\beta-\alpha)}.
$$
\end{cor}
\textbf{Proof}.
We have
${\displaystyle \mathbb T=\cup_{n}\omega_n\cup(\{\alpha<u<\beta\}\cap{\rm Supp}(u''+u))}\cup F_\alpha\cup F_\beta$, where $F_\alpha:=\{u=\alpha\}$, $F_\beta:=\{u=\beta\}$ and $\omega_n\subset\{\alpha <u <\beta\}$ open interval with $u''+u=0$ in $\omega_n$.
As $u'=0$ a.e. in $F_\alpha\cup F_\beta$ and $|\{\alpha<u<\beta\}\cap{\rm Supp}(u''+u)|=0$, it's enough to estimate $u'$ only in $\omega_n$. From $u''+u=0$ in $\omega_n$ we get $|u'|^2 + u^2 = \gamma^2$ with 
$\alpha^2 \leq \gamma^2 \leq \beta^2$. Therefore
$|u'|^2 = \gamma^2 - |u|^2 \leq 2\beta (\beta-\alpha)$,
which proves the statement.
\qed
\begin{remark}\rm
In Theorem \ref{th:concave} we have to work in an open interval $I$ of $\{a< u_0<b\}$ as, at this stage, it is not true in general that $S_{u_0}\cap\{a\leq  u_0 \leq b\}$ is finite (see Section \ref{sect:sharpness}). 
This property will be proved later with extra assumptions on $G$ at the boundary (see the proofs of Theorems \ref{th:polygon-g} and \ref{th:polygon-g-bis}).
\end{remark}
\begin{remark}\rm\label{no-nodes}
Assume that $\omega\subset\overline{\omega}\subset\{a<u_0<b\}$, with $\omega$ an open connected set, and that $n_\omega=\#\{\theta_n\in\omega\}\geq 3$, with $n\to\theta_n$ increasing.
Consider $v\in H_0^1(\theta_1,\theta_3)$  satisfying $v''+v=\delta_{\theta_2}$.
The function $v$ is admissible for the second order necessary condition. Similarly to the
case (a) we find the following estimation:
$$\frac{\theta_3-\theta_1}{\pi}\geq \frac{K_{pp}}{K_{up}+\sqrt{K_{up}^2 + K_{uu}K_{pp}}}
=:C(G,a,b),$$
Therefore, we get
\begin{equation*}
\#\{\theta_n\in\omega\}\leq 2\left[\frac{2\pi}{C(G,a,b)}\right]+1,
\end{equation*}
where $[\cdot]$ denotes the floor function.
\end{remark}
\begin{remark}\rm
 Theorem \ref{th:concave} and its proof are valid for non integral operators: if $j(u)=g(u,u')$ with $g:(u,p)\in W^{1,\infty}(\T)\times L^\infty(\T)\mapsto g(u,p)\in\R$, of class $\C^2$ and satisfying
 $$|g_{uu}(u_0,u_0')(v,v)|\leq K_{uu} \|v\|_{L^\infty}^2,\;|g_{up}(u_0,u_0')(v,v')|\leq K_{up} \|v\|_{L^\infty}\|v'\|_{L^2},\;g_{pp}(u_0,u_0')(v',v')\leq -K_{pp} \|v'\|_{L^2}^2$$
 for some $K_{uu},K_{up},K_{pp}>0$, the main argument \eqref{eq:main} still works (with a more precise Poincar\'e inequality, valid in dimension 1, namely $\|u\|_{L^\infty(0,\eps)}\leq \sqrt{\eps}\|u'\|_{L^2(0,\eps)},\;\forall u\in H^1_0(0,\eps)$).
\end{remark}

\subsection{Proof of Theorem \ref{th:concave}, case of volume constraint}\label{ssect:concave+m}
First, we point out that as $0<u_0\in H^1(\mathbb T)$, we may assume that there exist $0<a<b$ such that
$a< u_0< b$. Therefore, similarly to the case of inclusion in the annulus (see Remark \ref{r:tildeG}), we introduce a cut-off function to get a new $\widetilde{G}$ and a new functional $\widetilde{j}$, which is equal to $j$ on $\{u\in H^1(\T)\;;\; a< u <b\textrm{ and }|u'|\leq C(b)\}$ and therefore,
any solution of the problem (\ref{eq:min-g+m}) is still solution of
\begin{equation}
\min \left\{ \widetilde{j}(u),\quad u\in H^1(\mathbb T),\quad a<u<b,\quad u''+u\geq 0,\quad m(u)= m_0
\right\}.
\end{equation}
We can apply Proposition \ref{prop:euler+m} and write first and second order necessary conditions for the function $\widetilde{j}$, in terms of $\widetilde{G}$ (the constraint $a<u<b$ does not appear in the optimality condition, because these constrains are not saturated).
It is easy to check that $\widetilde{G}$ still satisfies the hypothesis in Theorem \ref{th:concave}. In the following, we denote by $j$, resp. $G$, the function $\widetilde{j}$, resp. $\widetilde{G}$.\\

\noindent
Now, we assume by contradiction that $u_0$ does not satisfy the theorem. Therefore there exits at least one accumulation point $\theta_0$ of $S_{u_0}$.
Without loss of generality we can assume $\theta_0=0$, and that there exists a decreasing sequence $\{\eps_n>0\}$ tending to $0$ such that
$S_{u_0}\cap (0,\eps_n)\neq\emptyset$. Then we can always find
$0<\eps_n^i<\eps_n$, $i=1,\ldots,5$, decreasing with respect to $i$, such that
$S_{u_0}\cap(\eps_n^{i+1},\eps_n^i)\neq\emptyset$, $i=1,4$. We consider
$v_{n,i}$ solving
\[
 v_{n,i}''+v_{n,i}=\chi_{(\eps_n^{i+1},\eps_n^i)}(u_0''+u_0),\quad
v_{n,i}=0 \mbox{ in } (0,\eps_n)^c,\quad i=1,4.
\]
Next, we extend the same idea of \cite{LP01New} that we used in the first part of the proof (section \ref{ssect:concave}) as follows: we look for $\lambda_{n,i},\ i=1,4$ such that ${\displaystyle v_n=\sum_{i=1,4}\lambda_{n,i} v_{n,i}}$ satisfies
\[
v'_n(0^+)=v'_n(\eps_n^-)=m'(u_0)(v_n)=0.
\]
Note that the derivatives at $0^+$ and $\eps_n^-$ are well defined as $v_{n,i}$ are regular
nearby $0$ and $\eps_n$ in the interval $(0,\eps_n)$. Such a choice of $\lambda_{n,i}$ is always possible as $\lambda_{n,i}$ satisfy three linear equations. Moreover, $v_n$ is not zero since $S_{u_0}\cap (\eps_n^i,\eps_n^{i+1})\neq\emptyset$. Using \eqref{zeta_0-bis}, we get $\int_{\mathbb T}v_n(\zeta_0+\zeta_0'')=0$, which implies
\[
0=j'(u_0)(v_n)=\int_{\mathbb T}v_n(\zeta_0+\zeta_0'')=m'(u_0)(v_n).
\]
As $v_n''+v_n \geq \lambda (u_0''+u_0)$ for $\lambda\ll 0$, it follows that $v_n$ is eligible for the second order necessary condition. Then, using \eqref{eq:concave},
\begin{eqnarray*}
0&\leq&
j''(u_0)(v_n,v_n)=
\int_{\mathbb T}
\left(G_{uu}(\theta,u_0,u_0') + \frac{3\mu}{u_0^4}\right)v_n^2 + 
2 G_{up}(\theta,u_0,u_0')v_n v_n' + G_{pp}(\theta,u_0,u_0'){v'_n}^2 \\
&\leq &
\int_{\mathbb T}
\left(K_{uu} + \frac{3|\mu|}{a^4}\right)v_n^2 + 2K_{up}|v_n||v_n'| - K_{pp}|v_n'|^2 \\
&\leq&
(o(1) - K_{pp})\|v_n'\|^2_{L^2},
\end{eqnarray*}
with $o(1)\to0$ as $n\to\infty$, where we have used Poincar\'e's inequality in
$H_0^1(0,\eps_n)$ (see (\ref{eq:K_}) for the notation $K_{uu}$, $K_{up}$ and $K_{pp}$). As $n$ tends to $\infty$, the inequality $0\leq j''(u_0)(v_n,v_n)$ becomes impossible and this proves the theorem.\qed

\subsection{Proof of Theorem \ref{th:polygon-g}}
If $j$ satisfies the hypotheses of Theorem \ref{th:polygon-g}, we can apply Theorem \ref{th:concave} (see also Remark \ref{r:tildeG}). Therefore, it remains to prove the following result:
\begin{prop}\label{prop:boundary-g}
Under the assumptions of Theorem \ref{th:polygon-g}, the sets $\{u_0=a\}$ and $\{u_0=b\}$ are finite.
\end{prop}
\begin{proof}
Assume by contradiction there exists $\theta_0$ an accumulation point of 
$\{(u_0-a)(u_0-b)=0\}$.\\
\textbf{(a)} \textbf{First case : $u_0(\theta_0)=a$.} Without loss of generality we can assume that $\theta_0=0$ and that there exists a sequence $\{\eps_n>0\}$ of $S_{u_0}$ tending to $0$, 
with $u_0(\eps_n)=a$ and $S_{u_0}\cap(0,\eps_n)\neq\emptyset$. \\
(a.1) First subcase: assume by contradiction that there exists a sequence $\theta_n\in S_{u_0}\cap(0,\eps_n)$ such that $\theta_n\to\theta_0$ and $a<u_0(\theta_n)<b$.
As $\{\theta,\, a<u_0(\theta)<b\}$ is open, there exists an open connected set $\omega_n$,
$\theta_n\in\omega_n\subset\{a<u_0<b\}$, $diam(\omega_n)\to0$, $u_0(\partial\omega_n)=a$. 
Consider the function $v_n$ given by
$v_n\in H^1(\T)$, $v_n''+v_n=u_0''+u_0=\sum_{i=1}^{N_i} \alpha_i\delta_{\theta_n^i}$ in $\omega_n$ (where $N_i$ is finite), 
$v_n=0$ in $\omega_n^c$ (from Theorem \ref{th:concave}, $u_0''+u_0$ is a finite sum of Dirac masses in $\omega_n$). It follows that for $n$ 
large $v_n$ is admissible (again using \eqref{zeta_0}, and also that $u_0=a$ on $\partial\omega$) for Proposition \ref{prop:euler}, since $u_0''+u_0$ has some Dirac masses in $\partial\omega_n$. 
Then we can apply the second order necessary condition, as in (b), Section \ref{ssect:concave},  which leads to a contradiction, since $diam(\omega_n)$ is going to 0.\\~\\
\begin{minipage}{80mm}
\figpt 0:(0,0)
\figpt 1:(-50,86)

\figpt 3:(0,100)
\figpt 4:(25,94)
\figpt 5:(48,85)
\figpt 6:(72,69)
\figpt 7:(79,61)
\figpt 8:(90,40)
\figpt 9:(98,20)
\figpt 10:(99,10)
\figpt 11:(100,0)

\figpt 12:(15,40)
\figpt 13:(62,20)


\psbeginfig{}

\psset(width=0.3)
\psarccirc 0; 100(0,97)

\psset(width=0.8)
\psline[3,4,5,6]
\psarccirc 0; 100(37.6,43.6)
\psline[7,8,9,10,11]
\psset(width=0.2)
\psarccirc 0; 40(43.6,90)
\psarccirc 0; 60(11.5,37.6)
\psset(dash=8)
\psline[0,3]
\psline[0,4]
\psline[0,5]
\psline[0,6]
\psline[0,7]
\psline[0,8]
\psline[0,9]
\psline[0,10]
\psline[0,11]

\psendfig
\figvisu{\figBoxA}{\bf Figure 3: Case (a.1)}{
\figwritene 8:{$\theta_n^1$}(2)
\figwritene 4:{$\theta_{n-1}^1$}(2)
\figwritene 5:{$\theta_{n-1}^2$}(2)
\figwritee 11:{$\theta_0$}(2)
\figwriten 3:{$\eps_n$}(2)
\figwritew 3:{$\frac{1}{a}$}(15)
\figwrite [12]{$\omega_{n-1}$}
\figwrite [13]{$\omega_n$}
}
\centerline{\box\figBoxA}
~\\
\end{minipage}
\begin{minipage}{80mm}
\figpt 0:(0,0)
\figpt 1:(-50,86)

\figpt 3:(0,100)
\figpt 4:(25,97)
\figpt 5:(72,69)
\figpt 6:(79,61)
\figpt 7:(98,20)
\figpt 8:(99,10)
\figpt 9:(100,0)

\figpt 10:(10,105)
\figpt 11:(75,70)
\figpt 12:(17,39)
\figpt 13:(60,20)


\psbeginfig{}

\psset(width=0.3)
\psarccirc 0; 100(0,97)

\psset(width=0.8)
\psline[4,5]
\psline[6,7,8,9]
\psarccirc 0; 100(75.5,90)
\psarccirc 0; 100(37.6,43.8)
\psset(width=0.2)
\psarccirc 0; 40(43.6,75.5)
\psarccirc 0; 60(11.5,37.6)
\psset(dash=8)
\psline[0,3]
\psline[0,4]
\psline[0,5]
\psline[0,6]
\psline[0,7]
\psline[0,8]
\psline[0,9]

\psendfig
\figvisu{\figBoxA}{\bf Figure 4: Case (a.2)}{
\figwritee 7:{$\theta_n$}(2)
\figwritee 9:{$\theta_0$}(2)
\figwriten 3:{$\eps_n$}(2)
\figwritew 3:{$\frac{1}{a}$}(15)
\figwritee 10:{$F_a$}(2)
\figwritee 11:{$F_a$}(2)
\figwrite [12]{$\omega_{i-1}$}
\figwrite [13]{$\omega_i$}
}
\centerline{\box\figBoxA}
~\\
\end{minipage}
(a.2) Second subcase: $(0,\eps_n)=F_a\cup_i\omega_i$ with $F_a=\{u_0=a\}\cap(0,\eps_n)$ relatively closed and $\omega_i\subset (0,\eps_n)$ open intervals with $u_0(\partial\omega_i)=a$ and $u_0''+u_0=0$ in $\omega_i$. 
Let $v_n$ given by
\[
v_n''+v_n = -(u_0''+u_0)\quad\mbox{ in }\quad (0,\eps_n),\qquad
v_n=0\quad\mbox{ in }\quad (0,\eps_n)^c.
\]
We have $v_n>0$ on $(0,\eps_n)$: indeed, as $(u_0+v_n)''+(u_0+v_n)=0$ in $(0,\eps_n)$ (so $u_0+v_n$ represents a line), $u_0+v_n=u_0$ in $\partial(0,\eps_n)$ and $u_0$ represents a convex curve, it follows that $u_0< u_0+v_n$ on $(0,\eps_n)$ ($v_n\not\equiv 0$ because $S_{u_0}\cap(0,\eps_n)\neq\emptyset$).
Then for $n$ large and $t\geq 0$ small the function $u_n=u_0+tv_n$ satisfies $a\leq u_n\leq b$, $u_n''+u_n\geq0$ (we use that $u_0''+u_0$ has positive Dirac masses at $0$ and $\eps_n$). Therefore, we can use the first order inequality (see Remark \ref{Tu0})
$j'(u_0)(v_n)\geq0$, which gives 
\[
0\leq j'(u_0)(v_n)=\int_{\T} G_u(u_0,u_0')v_n + G_p(u_0,u_0')v_n'.
\]
If (ii) holds we have 
$\int_{F_a}G_u(u_0,u_0')v_n + G_p(u_0,u_0')v_n'\leq 0$ because $u_0=a$ and $u_0'=0$ a.e.  in $F_a$,  $G_u(a,0) \leq 0$ and $G_p(a,0)=0$ (as $p\to G_p(a,p)$ is odd). So, if one of (ii) conditions holds, we have
\[
0\leq j'(u_0)(v_n) \leq \sum_i \int_{\omega_i} G_u(u_0,u_0')v_n + G_p(u_0,u_0')v_n'.
\]
Note that we have 
$\int_{\omega_i} G_u(u_0,u_0')u_0' + G_p(u_0,u_0')u_0''=
\left[G(u_0,u_0')\right]_{\partial\omega_i}=0$, since \linebreak$u_0'(\partial^+\omega_i)=-u_0'(\partial^-\omega_i)$ (where $\omega_i=(\partial^-\omega_i,\partial^+\omega_i)$) and $G(a,\cdot)$ is even. Therefore, from
\[
v_n=\alpha_{n,i} u_0 + \beta_{n,i} u_0'\quad\mbox{\it in}\quad\omega_i,\qquad
\alpha_{n,i}=\frac{\int_{\omega_i}u_0 v_n}{\int_{\omega_i}u_0^2} > 0,\quad
\beta_{n,i}=\frac{\int_{\omega_i}u_0' v_n}{\int_{\omega_i}|u_0'|^2},
\]
we get that if (ii) holds then
\begin{equation}\label{g1}
0\leq j'(u_0)(v_n)\leq 
\sum_i \alpha_{n,i} \int_{\omega_i} G_u(u_0,u_0')u_0 + G_p(u_0,u_0')u_0'. 
\end{equation}
We now prove that  
\begin{equation}\label{vn->0}
v_n\to 0,\quad
\overline{u}_n\to a\quad \mbox{\it in $W^{1,\infty}(\mathbb T)$ as $n\to\infty$},
\end{equation}
where $\overline{u}_n=u_0$ in $(0,\eps_n)$ and $\overline{u}_n=a$ in $(0,\eps_n)^c$. 
Indeed, the statement for $\overline{u}_n$ follows from Corollary \ref{cor:borne-u'} because  we have $\|\overline{u}_n-a\|_{L^\infty}\to 0$ as $n\to\infty$ (from $|u_0(\theta) - a|\leq \sqrt{\eps_n}\|u'_0\|_{L^2}$ for $\theta\in(0,\eps_n)$ and $\overline{u}_n''+\overline{u}_n\geq 0$). Next, from 
$(\overline{u}_n+v_n)''+(\overline{u}_n+v_n)=0$ in $(0,\eps_n)$ and
$\overline{u}_n+v_n=a$ in $(0,\eps_n)^c$, using again Corollary \ref{cor:borne-u'}, we find out that
$\|(\overline{u}_n+v_n)-a\|_{W^{1,\infty}(\mathbb T)}\to0$, which proves the statement for $v_n$.

Assume (ii.1) holds. 
We have $\overline{u}_n\to a$ in $W^{1,\infty}(\T)$ as $n\to\infty$, so $G_p(u_0,u_0')=o(1)$ as 
$n\to\infty$, and then \\
\[
0\leq j'(u_0)(v_n)\leq 
\sum_i \alpha_{n,i} \int_{\omega_i}\left(G_u(a,0)u_0 + o(1)\right),
\]
which is impossible as $n\to\infty$ because $G_u(a,0)<0$ and $\alpha_{n,i}>0$.\\
Now assume (ii.2) holds. In this case, we need a second order information: for $n$ large we have 
\begin{eqnarray*}
0&\leq &
j(u_0+v_n) - j(u_0)= j'(u_0)(v_n)+\frac{1}{2}j''(\widetilde{u}_n)(v_n,v_n)\\
&=&\int_0^{\eps_n} G_u(u_0,u_0')v_n + G_p(u_0,u_0')v_n' \\
&&+ \frac{1}{2}\int_0^{\eps_n} G_{uu}(\widetilde{u}_n,\widetilde{u}_n')v_n^2 + 
                                  2G_{up}(\widetilde{u}_n,\widetilde{u}_n')v_nv_n' + 
				  G_{pp}(\widetilde{u}_n,\widetilde{u}_n')|v_n'|^2\\
&\leq&\sum_i \alpha_{n,i} \int_{\omega_i} G_u(u_0,u_0')u_0 + G_p(u_0,u_0')u_0'\\
&&+ \frac{1}{2}\int_0^{\eps_n} (o(1) -\widetilde{K}_{pp})|v_n'|^2.
\end{eqnarray*}
Here $\widetilde{u}_0=u_0+\sigma_n v_n$, $\widetilde{u}_n'=u_0'+\sigma_n v_n'$ with a certain $\sigma_n\in(0,1)$, and we used the estimation (\ref{g1}) for $j'(u_0)(v_n)$, which holds as it uses only the fact $G_u(a,0)\leq 0$, and $G_{pp}(\widetilde{u}_n,\widetilde{u}_n')\leq-\widetilde{K}_{pp}<0$. The existence of $\widetilde{K}_{pp}>0$ follows from hypothesis (i), continuity of $G_{pp}$ at $(a,0)$ and the 
$W^{1,\infty}(\mathbb T)$ convergence in (\ref{vn->0}).
From (ii.2)  we have 
$\int_{\omega_i} G_u(u_0,u_0')u_0 + G_p(u_0,u_0')u_0'\leq 0$ and therefore we get
\[
0\leq 
j(u_0+v_n) - j(u_0)
\leq \frac{1}{2}\int_0^{\eps_n} (o(1)-\widetilde{K}_{pp})|v_n'|^2,
\]
which is impossible for $n$ large and proves that this case is cannot occur.\\
~\\
\textbf{(b) } \textbf{Second case }: 
$u_0(\theta_0)=b$. 
Without loss of generality we may assume $\theta_0=0$ and that there exists a sequence 
$\eps_n>0$ decreasing and tending to $0$ such that 
$u_0(2\eps_n)=b$. From Theorem \ref{th:concave}, it follows that  
$(0,2\eps_n)=\cup_{i\in N_n} \omega_{n,i}\cup \{\theta_n^i,\ i\in N_n\}\cup F_b$ with $F_b=\{u=b\}\cap(0,2\eps_n)$ relatively closed, $N_n\subset\N\cup\{\infty\}$, 
and $u_0''+u_0=0$ in the open intervals $\omega_{n,i}$ (see Figure 5).\\
\begin{minipage}{100mm}
Consider the function $u_n\in H^1(\mathbb T)$ given by
\[
\begin{array}{lll}
u_n=u_0 &\mbox{in}&(0,2\eps_n)^c,\\
u_n=b\cos\theta&\mbox{in}& (0,\eps_n),\\
u_n=b\cos(\theta-2\eps_n)&\mbox{in}&(\eps_n,2\eps_n),
\end{array}
\]
Let $\sigma_n=\sup\{\theta\in(0,\eps_n),\ u_0(\theta)=u_n(\theta)\}$, 
$\tau_n=\inf\{\theta\in(\eps_n,2\eps_n),\ u_0(\theta)=u_n(\theta)\}$. 
We have $u_0=u_n$ in $(0,\sigma_n)\cup(\tau_n,2\eps_n)$.\\
From the assumption of accumulation point, we must have $\sigma_n<\eps_n<\tau_n$. 
\end{minipage}
\begin{minipage}{60mm}
\figpt 0:(0,0)
\figpt 1:(-50,86)

\figpt 3:(56.5,56.5)
\figpt 4:(21,77.5)
\figpt 5:(-3,112)
\figpt 6:(-15,85.2)
\figpt 7:(-36,77)
\figpt 8:(-56.5,56.5)
\figpt 9:(53,63)
\figpt 10:(10,105)
\figpt 11:(75,70)
\figpt 12:(33,40)
\figpt 13:(70,20)
\figpt 14:(50,62)
\figpt 15:(6,47)


\psbeginfig{}

\psset(width=0.3)
\psarccirc 0; 45(75,100)

\psarccirc 0; 80(36,135)
\psset(width=0.8)
\psline[4,6,7,8]
\psarccirc 0; 80(45,75)

\psset(dash=8)
\psset(width=0.8)
\psline[3,5,7]
\psset(width=0.2)
\psline[0,3]
\psline[0,4]
\psline[0,5]
\psline[0,6]
\psline[0,7]
\psline[0,8]
\psline[0,14]

\psendfig
\figvisu{\figBoxA}{\bf Figure 5: Case (b)}{
\figwriten 9:{$\sigma_n$}(4)
\figwritene 3:{$0$}(2)
\figwritese 3:{$\frac{1}{b}$}(15)
\figwritese 4:{$F_b$}(13)
\figwriten 5:{$\eps_n$}(2)
\figwritene 6:{$\theta_{n}^i$}(1)
\figwritenw 7:{$\tau_n$}(2)
\figwritew 8:{$2\eps_n$}(2)
\figwriten 15:{$\omega_n^i$}(2)
}
\centerline{\box\figBoxA}
\end{minipage}
Besides, we have
\begin{equation}\label{u0<un}
0<u_n < u_0,\quad |u_0'| < |u'_n| \quad 
\mbox{\it a.e. in }\quad (\sigma_n,\tau_n).\\
\end{equation}

\noindent
The first inequality is clear. For the other inequality we point out that $0=u_0'<|u'_n|$ a.e. in $F_b$, and 
$|u'_n|^2+u_n^2 = b^2$, $|u'_0|^2 + u_0^2= c^2$ in $\omega_{n,i}\cap(\sigma_n,\tau_n)$, for some $c$ with $b^2\geq c^2$. Therefore
\[
|u'_n|^2 - |u_0'|^2 = b^2 - c^2 + u_0^2 - u_n^2 >0\quad \mbox{ in }\quad \omega_{n,i}\cap(\sigma_n,\tau_n).
\]
We also note that as in the case (a.2), 
$u_n\to b$ in $W^{1,\infty}(\mathbb T)$.
As $u_n$ satisfies $a\leq u_n\leq b$, $u_n''+u_n\geq 0$, and
$p\to G(u,p)$ is even near $(b,0)$ we get
\begin{eqnarray*}
0&\leq&
j(u_n)-j(u_0)=\int_0^{\eps_n}  G(u_n,|u_n'|) - G(u_0,|u_0'|) \\
&=&
\int_{\sigma_n}^{\tau_n}
\Big(G(u_n,|u_n'|) - G(u_n,|u_0'|)\Big) + \Big(G(u_n,|u_0'|)-G(u_0,|u_0'|)\Big) \\
&=&
\int_{\sigma_n}^{\tau_n}
(|u_n'|-|u_0'|)G_p\big(u_n,|u_0'|+t(|u_n'|-|u_0'|)\big) + (u_n-u_0)G_u\big(u_0 + s(u_n-u_0),|u_0'|\big)d\theta,
\end{eqnarray*}
with $0<t,s<1$. But from the parity of $p\mapsto G(\cdot,p)$ and $G_{pp}<0$ near $(b,0)$, it follows that $G_p(\cdot,p)<0$ for $p>0$ near $(b,0)$. Then from the assumption $G_u\geq0$ near $(b,0)$ the last inequality leads to a contradiction, so this case is impossible.
\qed
\end{proof}

\begin{remark}\rm\label{rk:theta}
Theorem \ref{th:polygon-g} can be extended to more general integral operators. More precisely, let $j(u)=\int_{\mathbb T}G(\theta,u,u')$ for some $G$ satisfying  \\
(i) $G$ is a $\mathcal{C}^2$ function, $p\mapsto G(\theta,u,p)$ is even and 
$G_{pp}(\theta,u_0,u'_0)<0$, $\forall\;\theta\in\mathbb T$,\\
(ii) $G_\theta(\theta,a,p)=0$ and $G_u(\theta,a,0)< 0$,\ for all $\theta\in\mathbb T$,\\
(iii) $G_u(\theta,u,p)\geq0$ near $(\theta,b,0)$, for all $\theta\in\mathbb T$,\\
where $u_0$ is a solution of problem (\ref{eq:min-g}). Then $S_{u_0}$ is finite, i.e. $\Omega_{u_0}$ is a polygon.

The proof of this results is very similar to the proof of Theorem \ref{th:polygon-g}, except for the analysis on the boundary $\{u_0=a\}$, which requires certain particular estimations.
\end{remark}

\subsection{Proof of Theorem \ref{th:polygon-g-bis}}
Conditions of Theorem \ref{th:concave} are satisfied, so it's enough to prove:
\begin{prop}\label{prop:concave-bis}
Assume the conditions (i), (ii) of Theorem \ref{th:polygon-g-bis} hold.
Then, for any solution $u_0$ of (\ref{eq:min-g}), and for $I=(\gamma_1,\gamma_2)\subset\{a<u_0<b\}$, there exists $n_0\in\N$ such that
\[
u_0+u_0''=\sum_{1\leq n\leq n_0}\alpha_n\delta_{\theta_n}\quad\mbox{in}\quad I
,\quad \alpha_n>0.
\]
\end{prop}
\begin{proof}
The proof follows closely the one of Theorem \ref{th:concave}. In fact the proof of 
steps (a) and (c) are identical, since we have 
$G_{pp}(u_0,u_0') \leq- K_{pp}(\alpha)< 0$ if $u_0\geq a+\alpha$, $\alpha>0$. Let us deal with the step (b), which needs a new proof.\\
\textbf{(b)}
Assume by contradiction that there exists $\theta_0$ an accumulation point of $S_{u_0}\cap I$ with 
$u_0(\theta_0)=a$ (see Figure 2). Without restriction we may take $\theta_0=0$ and assume there exists a decreasing sequence $\{\theta_n>0\}$ tending to $0$ such that 
$u_0''+u_0=\sum_{n\in \N}\alpha_n\delta_{\theta_n}$ and $u_0>a$ in $\{0<\theta\ll 1\}$ and $\alpha_n>0$. 
Like in \cite{LP01New}, we consider $v_n\in H_0^1(\mathbb T))$ given by
\[
 0\leq v_n(\theta)=\left\{
\begin{array}{lll}
\sin(\theta-\theta_{n+1})\sin(\theta_{n-1}-\theta_n)& in & (\theta_{n+1},\theta_n),\\
\sin(\theta_n-\theta_{n+1})\sin(\theta_{n-1}-\theta)& in & (\theta_n,\theta_{n-1}),\\
0,& in & (\theta_{n+1},\theta_{n-1})^c.
\end{array}
\right.
\]
Since $u_0''+u_0$ has some Dirac mass at $\{\theta_{n+1},\theta_{n},\theta_{n-1}\}$, and $u_0>a$ in $\{0<\theta\ll 1\}$, the function $v_n$ is admissible for the first and second order necessary conditions of
Proposition \ref{prop:euler}.  
From the first order condition we get
\begin{eqnarray*}
0&=&
\int_{\T} G_u(u_0,u_0')v_n + G_p(u_0,u_0')v'_n\\
&=&-\left[G_p(u_0,u_0')v_n\right]_{\theta_n}+\int_{\T\setminus{\theta_n}} \left(G_u(u_0,u_0')-\frac{d}{d\theta}G_p(u_0,u_0')\right)v_n\\
&=&
-[G_p(u_0,u_0')]_{\theta_n}v_n(\theta_n) + 
\int_{\T\setminus{\theta_n}} \Big(G_u(u_0,u_0') + G_{pp}(u_0,u_0') u_0 - G_{up}(u_0,u_0')u_0'\Big)v_n,
\end{eqnarray*}
since $u_0''+u_0=0\;\;on\;\;(\theta_{n+1},\theta_{n-1})\setminus\{\theta_n\}$ ($[\cdot]_\theta$ denotes the jump at $\theta$).\\
We now prove the following consequence:
\begin{equation}\label{eq:Gu-Gup=0}
G_u(a,u_0'(0^+)) - G_{up}(a,u_0'(0^+)) u_0'(0^+) = 0.
\end{equation}
We will prove (\ref{eq:Gu-Gup=0}) using the technique used in \cite{LP01New} for a particular functional 
$G(u,p)$. First we point out that
\begin{eqnarray*}
&&\lim_{n\to\infty}
\frac{\int_{\T} (G_u(u_0,u_0') + G_{pp}(u_0,u_0') u_0 - G_{up}(u_0,u_0')u_0')v_n}
{\int_{\T}v_n}\\
&=&G_u(a,u_0'(0+)) - G_{up}(a,u_0'(0^+))u_0'(0^+) 
=\lim_{n\to\infty}\frac{[G_p(u_0,u_0')]_{\theta_n}v_n(\theta_n)}{\int_{\T}v_n}\leq 0,
\end{eqnarray*}
where we have used that fact that $p\to G_p(u,p)$ is decreasing (consequence of $G_{pp}\leq 0$), $G_{pp}(a,p)=0$ and $[u'_0]_{\theta_n}>0$.\\
If by absurd (\ref{eq:Gu-Gup=0}) does not hold, there exists a constant $c>0$ such that 
\begin{equation}\label{eq:>c}
-\frac{[G_p(u_0,u_0')]_{\theta_n}v_n(\theta_n)}{\int_{\T}v_n} \geq c >0
\end{equation}
for $n$ large. Since $G_{pp}(a,\cdot)=0$ we have
\begin{eqnarray*}
[G_p(u_0,u_0')]_{\theta_n}
&=& G_p(u_0(\theta_n),u_0'(\theta_n^+)) - G_p(u_0(\theta_n),u_0'(\theta_n^-)) \\
&=&[u_0']_{\theta_n} G_{pp}(u_0(\theta_n),\widetilde{u}_{0n}')
 = [u_0']_{\theta_n} (G_{pp}(u_0(\theta_n),\widetilde{u}_{0n}')-G_{pp}(u_0(0),\widetilde{u}_{0n}'))\\
&=&\theta_n[u_0']_{\theta_n} \int_0^1 G_{upp}(u_0(t\theta_n),\widetilde{u}_{0n}')u'_0(t\theta_n)dt,
\end{eqnarray*}
with $\widetilde{u}_{0n}'$ between $u_0'(\theta_n^+)$ and $u_0'(\theta_n^-)$. We point out that
${\displaystyle
\frac{\int_{\T} v_n}{\theta_n v_n(\theta_n)} = 
\frac{1}{2}\frac{\tau_n + \tau_{n-1}}{\sum_{j=n}^\infty\tau_j}(1 + o(1))} 
$
and the series
${\displaystyle
\sum_n\frac{\tau_n + \tau_{n-1}}{\sum_{j=n}^\infty\tau_j}=+\infty}$,
where $\tau_k = \theta_{k} - \theta_{k+1}$, (from an elementary lemma on series, see \cite{LP01New}). Therefore, from (\ref{eq:>c}) 
we obtain 
\begin{equation*}
-\frac{[G_p(u_0,u_0')]_{\theta_n}v_n(\theta_n)}{\int_\T v_n}
= - [u_0']_{\theta_n} \left(\int_0^1 G_{upp}(u_0(t\theta_n),\widetilde{u}_{0n}')u'_0(t\theta_n)dt\right) \frac{\theta_nv_n(\theta_n)}{\int_\T v_n}\geq c.
\end{equation*}
As $\int_0^1 G_{upp}(u_0(t\theta_n),\widetilde{u}_{0n}')u'_0(t\theta_n)dt$ is uniformly bounded w.r.t. to $n$, with a summation, we get:
\begin{eqnarray*}
\infty >C \sum_n [u_0']_{\theta_n}&\geq& - \sum_n [u_0']_{\theta_n} \left(\int_0^1 G_{upp}(u_0(t\theta_n),\widetilde{u}_{0n}')u'_0(t\theta_n)dt\right) \\
&\geq& \sum_n c\frac{\int_\T v_n}{\theta_nv_n(\theta_n)}
\geq
\sum_n \frac{c}{2}\frac{\tau_n + \tau_{n-1}}{\sum_{j=n}^\infty\tau_j}(1+o(1)),\\ 
&=& \infty.
\end{eqnarray*}
The contradiction proves (\ref{eq:Gu-Gup=0}). The important corollary of (\ref{eq:Gu-Gup=0}) is 
\begin{equation}\label{Gupp<0}
u_0'(0^+)>0,\quad G_{upp}(a,u_0'(0+))<0. 
\end{equation}
Indeed, from (\ref{eq:Gu-Gup=0}) and (ii) it follows that 
$0\neq G_{up}(a,u_0'(0^+)) u_0'(0^+)<0$. 
As $u_0(0)\leq u_0(\theta)$ implies $u_0^+(0)\geq 0$, it follows that $u_0'(0^+)>0$ and $G_{up}(a,u_0'(0^+))<0$. Using once more (ii) gives
\[
0>G_u(a,u_0'(0^+))=G_{up}(a,u_0'(0^+)) u_0'(0^+) = 
z(a,u_0'(0^+)) G_{upp}(a,u_0'(0^+)),
\]
which proves (\ref{Gupp<0}). 

Using $v_n$ in the second order condition of Proposition \ref{prop:euler} gives
\begin{eqnarray}
0&\leq&
\int_{\theta_{n+1}}^{\theta_{n-1}} 
G_{uu}(u_0,u_0')v_n^2 + G_{up}(u_0,u_0')(v_n^2)'+G_{pp}(u_0,u_0') |v_n'|^2 	\nonumber \\
&=&
-[G_{up}(u_0,u_0')]_{\theta_n} v_n(\theta_n)^2 			\nonumber\\
&&+
\int_{\theta_{n+1}}^{\theta_{n-1}} 
\left[
G_{uu}(u_0,u_0') - G_{uup}(u_0,u_0')u_0' + G_{upp}(u_0,u_0')u_0
\right]v_n^2+
G_{pp}(u_0,u_0') |v_n'|^2				\nonumber \\
&\sim&
o(1)\tau_n^2\tau_{n-1}^2 + 
\int_{\theta_{n+1}}^{\theta_{n-1}} G_{pp}(u_0,u_0')|v_n'|^2.	\label{g2}
\end{eqnarray}
Since $G_{pp}(a,0)=0$, we need further developments allowing to use (\ref{Gupp<0}). Namely
\begin{eqnarray*}
G_{pp}(u_0,u_0')
&=& G_{upp}(a,u_0')u_0'(0^+)\theta(1 + o(1)),\\
\int_{\theta_{n+1}}^{\theta_{n-1}} G_{pp}(u_0,u_0') |v_n'|^2
&=& \int_{\theta_{n+1}}^{\theta_{n-1}} G_{upp}(a,u_0')u_0'(0^+) \theta|v_n'|^2(1+o(1))\\
&=& u_0'(0^+)G_{upp}(a,u_0'(0^+))\int_{\theta_{n+1}}^{\theta_{n-1}}\theta|v_n'|^2(1 + o(1))\\
&\sim& u_0'(0^+)G_{upp}(a,u_0'(0^+)) 
    (\tau_n^2\tau_{n-1}^2 + \theta_{n+1}\tau_n\tau_{n-1}^2 + \theta_n\tau_n^2\tau_{n-1}).
\end{eqnarray*}
From (\ref{Gupp<0}), the last inequality contradicts the second order condition (\ref{g2}) and proves that this case is impossible.\qed
\end{proof}

\begin{prop}\label{prop:boundary-g-bis}
Under the assumptions of Theorem \ref{th:polygon-g-bis} the sets $\{u_0=a\}$ and $\{u_0=b\}$ are finite.
\end{prop}
\begin{proof} 
The proof of proposition follows closely the proof of Proposition \ref{prop:boundary-g}, except for the case (a.1) which needs another proof as $G_{pp}(u,p)$ is not strictly negative near $u=a$. Note that the case (a.2) of Proposition \ref{prop:boundary-g} when using only condition (ii.1) (which is the case in this proposition) does not require $G_{pp}<0$
(but only $G_u(a,0)<0$ and the parity of $p\to G(a,p)$). Furthermore, the case (b) of Proposition \ref{prop:boundary-g} requires only the (even) parity of $p\to G(u,p)$, $G_{pp}(u,p)<0$ and $G_u\leq 0$ near $(b,0)$.

\noindent
{\bf (a.1)} We assume by contradiction that 0 is an accumulation point of 
$S_{u_0}\cap\{u_0=a\}$, and that there exists a sequence $\{\eps_n>0\}$ tending to $0$, 
with $u_0(\eps_n)=a$ and $S_{u_0}\cap(0,\eps_n)\cap\{a<u_0<b\}\neq\emptyset$ (see Figure 3). Then, there exists an open interval $\omega_n\subset(0,\eps_n)\cap\{a<u_0<b\}$, with $S_{u_0}\cap\omega_n\neq\emptyset$ and $u_0(\partial\omega_n)=a$.
From Theorem \ref{th:concave} it follows that
$S_{u_0}\cap\omega_n$ is finite. Therefore, we can denote $\omega_n=(\theta_{n+1},\theta_{n-1})$ and find $\theta_n\in (\theta_{n+1},\theta_{n-1})\cap S_{u_0}$. We then consider
\[
 0\leq v_n(\theta)=\left\{
\begin{array}{lll}
\sin(\theta-\theta_{n+1})\sin(\theta_{n-1}-\theta_n)& in & (\theta_{n+1},\theta_n),\\
\sin(\theta_n-\theta_{n+1})\sin(\theta_{n-1}-\theta)& in & (\theta_n,\theta_{n-1}),\\
0,& in &  (\theta_{n+1},\theta_{n-1})^c.
\end{array}
\right.
\]
The function $v_n$ is admissible for the first order condition, since $u_0''+u_0$ has some positive Dirac mass on $\partial\omega_n$. We can proceed exactly as in step (b) of Proposition \ref{prop:concave-bis} and we prove that (\ref{Gupp<0}) holds, so $u_0'(0^+)>0$. However, from the fact that 
$\theta_0=0$ is an accumulation point from the right, it's easy to show that $u_0'(0^+)=0$. The contradiction proves the claim.\qed 
\end{proof}

\section{Sharpness of conditions}\label{sect:sharpness}
The conditions of Theorem \ref{th:polygon-g}, \ref{th:polygon-g-bis} are optimal in the sense that there exist counterexamples with $G(u,u')$ not satisfying one of (i)-(iii) and such that the corresponding solution of (\ref{eq:min-g}) is not a polygon. We will provide some counterexamples for Theorems \ref{th:polygon-g}, \ref{th:polygon-g-bis}.

\subsection{Counterexamples for Theorem \ref{th:polygon-g}}\label{conds-g}
\noindent
{\bf Condition (i)}\\
Set $c=(a+b)/2$ and consider $G(u,p)=\frac{1}{2}\left((u - c)^2 + p^2\right)$.
Note that $G$ satisfies (ii.1) as \linebreak[4]$G_u(a,0) =a-c<0$ and (iii) because $G_u(b,0)=b-c>0$.
It does not satisfy (i) because $G_{pp}=1$.  It is obvious that the corresponding
solution of (\ref{eq:min-g}) is not a polygon, but rather the circle $\{u_0=c\}$.\\

\noindent
{\bf Condition (ii)}\\
Consider the function $G(u,p)=\frac{1}{2}(u^2-p^2)$. Of course $G_u(u,p)=u$ and $G_{pp}(u,p)=-2$, so $G(u,p)$ satisfies the conditions (i) and
(iii), but it does not satisfy (ii.1), neither (ii.2). The solution of (\ref{eq:min-g}) corresponding to
this $G(u,p)$ is the circle $u_0=a$. Indeed, for  admissible $u$ we have
\[
j(u)=\frac{1}{2}\int_{\mathbb T} (u^2-|u'|^2) = \frac{1}{2}\int_{\mathbb T} (u+u'')u
\geq\frac{a}{2}\int_{\mathbb T}(u+u'') = \frac{a}{2}\int_{\mathbb T}u \geq \pi a^2 = j(u_0)
\]
which proves that $u_0\equiv a$ is the minimizer of $j(u)$.\\

Another counterexample is using the perimeter. Indeed, if $G(u,p)=-\frac{(u^2+p^2)^{1/2}}{u^2}$ then \linebreak[4]$j(u):=\int_{\mathbb T}G(u,u')d\theta=-P(u)$, where $P(u)$ is the perimeter of the domain inside the curve $\{(1/u(\theta),\theta),\ \theta\in\T\}$.  Therefore, solution of (\ref{eq:min-g}) is
$u_0\equiv a$, which corresponds to the circle $\{r=1/a\}$. On the other side, $G(u,p)$ satisfies the conditions (i) and (iii) but none of conditions (ii). Indeed,
\[
G_u(u,0) = \frac{1}{u^2},\quad G_{pp}(u,p) = - \frac{1}{(u^2+p^2)^{3/2}}.
\]

\noindent
{\bf Condition (iii)}\\
Set $G(u,p)=-\frac{1}{2}(u^2+p^2)$. Since $G_u=-u$ and $G_{pp}=-2$, $G(u,p)$ satisfies (i), (ii.1), but it does not satisfy (iii). A solution of the corresponding minimization problem is $u_0\equiv b$. In fact, any $u_0$ representing a convex polygon with edges tangent to the circle $\{u_0=b\}$ is a solution! We can also add some piece of circle in the boundary. Indeed, first let $v$ be a function such that $1/v$ represents a straight line with $v\leq b$. For such $v$, we have
\[
 v^2+|v'|^2 \leq b^2,
\]
because $v$ satisfies the equation $v+v''=0$, so $((v^2)+(v')^2)'=0$ and therefore
$v^2+|v'|^2=k^2$. For $\theta_0$ such that $v'(\theta_0)=0$ the value of $1/v(\theta_0)$ gives the distance of the origin from the line $v$, so we must have $1/v(\theta_0)\geq 1/b$, which proves the claim. \\
Now, every admissible $u$ can be approached for the $H^1(\mathbb T)$ norm by a sequence of convex polygons $u_n$ satisfying $a\leq u_n \leq b$. Then
\[
j(u)=\lim_{n\to\infty}j(u_n)=-\frac{1}{2}\lim_{n\to\infty}\int_{\mathbb T}(u^2_n+|u'_n|)^2
\geq - \pi b^2
=j(u_0),
\]
which proves that $u_0\equiv b$ is a minimizer. This example provides some optimal shapes having an infinite number of corners inside $\{a<u<b\}$ (because we can have an infinite number of edges, tangent to the circle of radius $1/b$).

\subsection{Counterexamples for Theorem \ref{th:polygon-g-bis}}\label{ssect:sharpness}
With minor modifications, the counterexamples given in (i), (ii) and (iii) above can easily be updated for Theorem \ref{th:polygon-g-bis}. \\

\noindent
{\bf Condition (i)}
Let $c=\frac{1}{2}(a+b)$ and $G(u,p)=\frac{1}{2}((u-c)^2 + (u-a)^2p^2)$.
The function $G$ satisfies the
(ii), (iii) of Theorem \ref{th:polygon-g-bis}. Indeed,
\begin{eqnarray*}
&(ii):& G_u(a,p)=a-c<0,\\
    &&pG_{up}(a,p)=0,\quad G_{upp}(a,p)=0,\mbox{ so } pG_{up}(a,0)=z(p)G_{upp}(a,p)
    \mbox{ with } z= 0.\\
&(iii):& G_u(b,0)=b-c>0.
\end{eqnarray*}
The condition (i) is not satisfied as $G_{pp}=2(u-a)^2$ (note that $G_{pp}(a,p)=0$).
For $u$ admissible we have
\[
j(u) \geq 0 = j(c),
\]
so $u_0\equiv c$ minimizes $j(u)$.\\

\noindent
{\bf Condition (ii)}
Let $G(u,p)$ and $j(u)$ be as in the first example of Condition (ii) of Section \ref{conds-g}. We consider
\[
 \widehat{G}(u,p)=\frac{1}{2}(u^2 - \varphi(u)p^2),\quad
0\leq \varphi\leq 1,\; \varphi\in\mathcal C^\infty(\R),\;
\varphi(u)=
\left\{
\begin{array}{ll}
0,&u\leq a,\\
1,&u\geq b.
\end{array}
\right.
\]
and let $\widehat{j}(u)=\int_{\T}\widehat{G}(u,u')$. The function $\widehat{G}$ satisfies the
(i), (iii) of Theorem \ref{th:polygon-g-bis}, but not (ii). For $u$ admissible we have
\[
 \widehat{j}(u)=\int_{\T}\widehat{G}(u,u')\geq
\int_{T}G(u,u')=j(u)\geq j(a)=\widehat{j}(a),
\]
so $u_0\equiv a$ minimizes $\widehat{j}(u)$.\\

\noindent
{\bf Condition (iii)}
Again, let $G(u,p)$ and $j(u)$ be as in the Condition (iii) of Section \ref{conds-g}.
We consider $\widehat{G}(u,p)=-\frac{1}{2}(u^2 +\varphi(u)p^2)$ and
$\widehat{j}(u)=\int_{\T}\widehat{G}(u,u')$. The function $\widehat{G}$ satisfies the
(i), (ii) of Theorem \ref{th:polygon-g-bis}, but not (iii). Similarly as above,
for $u$ admissible we have
\[
 \widehat{j}(u)\int_{\T}\widehat{G}(u,u')\geq
\int_{T}G(u,u')=j(u)\geq j(b)=\widehat{j}(b),
\]
so $u_0\equiv b$ minimizes $\widehat{j}(u)$. Same remarks as in the previous subsection can be done. We can construct some optimal shapes locally polygonal inside $\{a<u<b\}$ (necessary because of Proposition \ref{prop:concave-bis}), but having an infinite number of corners in $\{a<u<b\}$ (the only condition to be a minimizer is that every edges of these shapes are tangent to the circle of radius $1/b$, and inside the domain $\{\varphi=1\}$).\\

\noindent{\bf Acknowledgments :}\\
The two authors would like to thank professor Michel {\sc Pierre} for introducing them in this interesting subject and for some very helpful comments and discussions about this paper.

\bibliographystyle{plain}

\end{document}